\documentclass[reqno]{amsart}

\usepackage[utf8]{inputenc}
\usepackage[T1]{fontenc}

\usepackage{hyperref}
\usepackage{amsmath}
\usepackage{amssymb}
\usepackage{amsfonts}
\usepackage{graphicx}
\usepackage{amsthm}
\usepackage{enumerate}
\usepackage{lscape}
\usepackage{dsfont}
\usepackage{color}
\usepackage{mathtools}

\usepackage{setspace}
\newcommand{\R}{\mathds{R}}

\newcommand{\Ric}{\mathrm{Ric}}

\newcommand{\C}{\mathds{C}}            
\newcommand{\de}{\partial}          

\newcommand{\K}{K\"{a}hler}

\newcommand{\OO}{{\mathcal{O}}}

\newcommand{\ov}[1]{\overline{#1}}
\newcommand{\wi}[1]{\widetilde{#1}}

\newcommand{\ngh}{neighbourhood}

\newcommand{\W}{\Omega}

\newcommand{\va}{\varphi}
\newcommand{\vf}{\varphi}
\newcommand{\hs}{\hspace{0.1em}}

\newcommand{\FR}{{\mathfrak{R}}}
\newcommand{\FF}{{\mathfrak{F}}}
\newcommand{\CG}{{\mathcal{G}}}

\def\F{\wi{\mathcal F}}

\newtheorem{theor}{Theorem}[section]
\newtheorem{prop}[theor]{Proposition}

\begin{document}

\title[Holomorphic isometries into homogeneous bounded domains]{Holomorphic isometries into homogeneous bounded domains}

\author{Andrea Loi}
\address{(Andrea Loi) Dipartimento di Matematica \\
         Universit\`a di Cagliari (Italy)}
         \email{loi@unica.it}

\author{Roberto Mossa}
\address{(Roberto Mossa) 
Instituto de Matem\'atica e Estatistica  \\
         Universidade de S\~ao Paulo (Brasil)}
         \email{robertom@ime.usp.br}

\thanks{
The first author  supported  by Prin 2015 -- Real and Complex Manifolds; Geometry, Topology and Harmonic Analysis -- Italy, by INdAM. GNSAGA - Gruppo Nazionale per le Strutture Algebriche, Geometriche e le loro Applicazioni,  by STAGE - Funded by Fondazione di Sardegna and Regione Autonoma della Sardegna.\\
\hspace*{1.3ex}The second author was supported by
FAPESP (grant: 2018/08971-9)}


\begin{abstract}
We prove two rigidity  theorems  on holomorphic isometries into homogeneous bounded domains.
The first  shows that a \K-Ricci soliton induced by the homogeneous metric of a homogeneous bounded  domain is trivial, i.e. \K-Einstein. In the second one we prove that 
a homogeneous  bounded domain and the flat  (definite or indefinite) complex Euclidean space are not  relatives, i.e. they do not share a common \K\ submanifold (of positive dimension). Our  theorems extend  the results  proved in  \cite{LMpams} and \cite{Cheng2021}, respectively.
\end{abstract}
 
\maketitle

\tableofcontents  

\section{Introduction}
The complex $n$-dimensional   hyperbolic space, namely the complex unit ball $\C H^n\subset \C^n$ equipped with the hyperbolic metric $g_{hyp}$ (a multiple of the Bergman metric of  constant and negative  holomorphic sectional curvature) is  the  prototype of noncompact  and nonflat \K\ manifold.  Thus, many rigidity results of the \K\ geometry of the complex hyperbolic space have been the subject of study by various mathematicians.
 Of particular interest  is the  study of  its  \K\ submanifolds, or more generally, those \K\  manifolds $(M,  g)$ which admits a holomorphic isometry  into  $(\C H^n, g_{hyp})$. 
 The first result in this direction was obtained in the celebrated work of Calabi \cite{Cal} were it is proven (among other things) that 
 the complex  flat space does not  admit a holomorphic isometry into any complex hyperbolic space.
 This result has been later extended by Umehara as expressed by the following striking result.
 
 \vskip 0.3cm

\noindent
{\bf Theorem A (\cite{UmearaD})}
{\em If  $(M, g)$ admits a holomorphic isometry into the
complex hyperbolic space then it cannot admit a holomorphic isometry into 
the complex flat space.}

\vskip 0.3cm

The reader is referred to \cite{CDY} for the  extension of Theorem A to the case  of  indefinite complex hyperbolic and flat spaces.

Another  recent rigidity  result on holomorphic isometries into the complex hyperbolic space is the following theorem:

\vskip 0.3cm

\noindent
{\bf Theorem B (\cite{LMpams})}
{\em If  $(g, X)$ is a \K-Ricci soliton on a complex manifold $M$ 
and $(M, g)$ admits a holomorphic isometry into the 
complex hyperbolic space
then the soliton is trivial, i.e. $g$ is KE.}

\vskip 0.3cm

Recall that a \K-Ricci soliton (KRS)  $(g, X)$ (see e.g. \cite{LMpams} and  \cite{BG} for details and references)
 consists of   a \K\ metric $g$ and  a holomorphic vector field $X$, called  the {\em solitonic vector field} such that 
$\Ric_{g}=\lambda g+L_{X}g$
where $\Ric_{g}$ is the Ricci tensor of the metric $g$,  $\lambda$ a constant and $L_Xg$ denotes the Lie derivative of $g$ with respect to $X$.
Clearly KRS generalize KE metrics. Indeed any
KE metric $g$ on a complex manifold $M$ gives rise to a
trivial KRS by choosing $X = 0$ or $X$ Killing with
respect to $g$.

Notice that Umehara \cite{UmearaE}  proves that a holomorphic isometry of a KE  manifold into $(\C H^n, g_{hyp})$ is necessarily totally geodesic, and hence the conclusion of Theorem B  is that $(M, g)$ is  an open subset of a complex hyperbolic space.

Recently Cheng and Hao \cite[Theorem 3]{Cheng2021} extend Theorem A to the case of the  Bergman metric on a homogeneous bounded  domain: 

\vskip 0.3cm

\noindent
{\bf Theorem C (\cite[Theorem 3]{Cheng2021}})
{\em Let $(M, g)$ be a  \K\ manifold  which admits a holomorphic   isometry  into  a homogeneous bounded domain
$\left(\Omega, g_{B}\right)$ endowed with its Bergman metric $g_B$. 
Then $(M, g)$ cannot admit a holomorphic isometry  into  a  complex Euclidian space.}

\vskip 0.3cm

The main result of the paper is the following Theorem \ref{mainteor} where it is shown  that  we get the same conclusion as in  Theorem B  and Theorem C if we replace 
either the  hyperbolic metric or the Bergman metric  by any homogeneous metric  $g_\Omega$ on a  
 homogeneous bounded  domain $\Omega$. 
 Recall that $g_\Omega$ is a homogenous metric on a bounded domain  $\Omega\subset \C^n$ 
 if $\left(\Omega, g_\Omega\right)$ is acted upon transitively by its group biholomorphic isometries.
Notice that  a  bounded symmetric domain with its Bergman metric  is a very special example  of this instance.
Moreover, there exist many  homogeneous  \K\ metrics  on a  bounded  domain   different form the Bergman metric which  
is KE  by the homogeneity assumption (see Theorem \ref{thmhide} below, \cite{kodomain} and  \cite{SIL} for details).

\begin{theor}\label{mainteor} Let $(M, g)$ be a  \K\ manifold  which admits a holomorphic   isometry  into  a homogeneous bounded domain
$\left(\Omega, g_\Omega\right)$. 
Then the following facts hold true:
\begin{itemize}
\item [(i)]
any  KRS on $M$ is trivial, i.e. $g$ is KE;
\item [(ii)]
$(M, g)$ cannot admit a holomorphic isometry  into  a  complex (definite or indefinite) Euclidian space.
\end{itemize}
\end{theor}

In \cite{diloi} the authors inspired by Umehara's work \cite{UmearaD}   have christened  two \K\ manifolds $(M_i, g_i)$, $i=1, 2$ to be {\em relatives} if there exists a \K\ manifold  $(M, g)$ and  two  holomorphic isometries  $\varphi_i:M\rightarrow M_i$, $i=1, 2$
(for update results on relatives \K\  manifolds  the reader is referred  the  the survey paper \cite{YUANYUAN}  
and reference therein). In this language we can rephrase  (ii) in Theorem \ref{mainteor} by saying  that a homogeneous bounded domain and the complex (definite or indefinite) flat space are not relatives (see also \cite{Mossa}).

It is also  worth  noticing  that a homogeneous bounded domain admits a holomorphic isometry into the  infinite dimensional complex projective space (see \cite{LMgeomded}).
Unfortunately this is not of  any help in the proof of (i) of Theorem \ref{mainteor} since there are  examples of non trivial KRS which admits holomorphic isometries
 into the infinite dimensional complex projective space (see \cite{LSZpac}).

The proof of Theorem \ref{mainteor} is based on  Theorem \ref{lem}  which is a transcendental result on holomorphic Nash algebraic functions
and on Theorem \ref{thmhide} (unpublished) due to Hishi Hideyuki which gives an explicit  description of the  structure of a  \K\ potential  of the  homogeneous \K\  metric $g_\Omega$
on a homogeneous bounded domain.

The paper contains three other sections. Sections \ref{secnash} and Section \ref{appendix}
 are dedicated to the proofs of Theorem \ref{lem} and Theorem \ref{thmhide}, respectively.
 Section \ref{main} contains the proof of Theorem \ref{mainteor}.
 
\vskip 0.3cm
The authors are indebted to Hishi Hideyuki for the proof of  Theorem \ref{thmhide}
and for the possibility of including it in our paper.

\section{A result on holomorphic Nash algebraic functions}\label{secnash}

Let $\mathcal N^m$ be the set real analytic function $\xi : V\subset \C^m \to \R $ defined in some open \ngh\ $V\subset \C^m$, such that its real analytic extension $\tilde\xi (z, w)$ in a \ngh\ of the diagonal of $V \times \operatorname {Conj} V$ is a holomorphic Nash algebraic function (for background material on Nash functions, we refer the readers to
\cite{XJHuang} and \cite{Tworz}). We define
$$
\mathcal F= \left\{\xi \left(f_1,\dots,f_m\right) \mid \xi \in \mathcal N^m,\ f_j \in \OO_0,\ j=1,\dots,m, \ m>0  \right\},
$$
where $\OO_0$ denotes the germ of holomorphic functions around $0\in \C$ and
we set 
\begin{equation}\label{tildeF}
\wi {\mathcal F}= \left\{\psi \in \mathcal F \mid \psi  \text{ is of diastasis-type} \right\}.
\end{equation}
Here we say (see also \cite{LMpams}) that  a real analytic function defined on a neighborhood   $U$ of   a point $p$
of a complex manifold $M$ is of  {\em diastasis-type} if 
in one (and hence any) coordinate system 
$\{z_1, \dots , z_n\}$ centered at $p$ its  expansion in $z$ and $\bar z$ 
does not contain non constant purely holomorphic or anti-holomorphic terms (i.e. of the form $z^{j}$ or $\bar{z}^{j}$ with  $j > 0$).

The key element in the proof  of Theorem \ref{mainteor} is the following Theorem \ref{lem}.  
This theorem  generalizes \cite[Theorem 2.1]{LMpams}.
Moreover its  proof can be considered an extension of the techniques used in the proof of  \cite[Theorem 1]{Cheng2021}.

\begin{theor}\label{lem}
Let $\psi_0\in \F\setminus \R$.  Then for every $\mu_1,\dots,\mu_\ell \in \R$ we have
\begin{equation}\label{fundeq}
e^{\psi_0}\notin \F^{\mu_1}\cdots \F^{\mu_\ell}\setminus \R
\end{equation}
where $ \F^{\mu_1}\cdots \F^{\mu_\ell}=\left\{\psi_1^{\mu_1}\cdots \psi_\ell^{\mu_\ell} \mid \psi_1,\dots,\psi_\ell \in \F\right\} $. 
\end{theor}


\begin{proof}[Proof of Theorem \ref{lem}]
Assume that 
\begin{equation}\label{egznew}
e^{\psi_0}=\psi_1^{\mu_1}\cdots \psi_\ell^{\mu_\ell} \in \F^{\mu_1}\cdots \F^{\mu_\ell}
\end{equation}
with 
\begin{equation}\label{egz}
\psi_k=\xi _k\left(f_{1}^{(k)},\dots,f_{m_k}^{(k)}\right)\in \F,\  k=0,\dots,\ell .
\end{equation}
Let us rename the functions involved in \eqref{egz} by
\begin{equation}\label{eqfvarf}
\left(\varphi_1, \dots , \varphi_s\right)=\left(f_{1}^{(0)},\dots,f_{m_0}^{(0)},\dots, f_1^{(\ell)},\dots,f^{(\ell)}_{m_\ell}\right)
\end{equation}
and let
$$
S=\left\{\varphi_1, \dots , \varphi_s\right\}.
$$ 

Let $D$ be an open neighborhood  of the origin of $\C$ on which each $\varphi_j$,  $j=1, \dots ,s$, is defined.
Consider   the field  $\mathfrak R$ of rational function on $D$ and its  field extension 
$
\mathfrak F = \mathfrak R \left(S\right),
$
namely, the smallest subfield of the field of the  meromorphic functions on $D$, containing  rational functions and the elements of $S$. Let $l$ be the transcendence degree  of the field extension $\FF/\FR$. 
If $l=0$, then each element in $S$ is holomorphic  Nash algebraic and hence $\psi_0$ is forced to be constant  by \cite[Lemma 2.2]{HYbook}. Assume then that $l>0$.
Without loss of generality we can assume that  $\CG=\left\{\varphi_1,\dots,\varphi_l\right\}\subset S$  is a maximal algebraic independent subset over $\FR$. Then   there exist minimal polynomials $P_j\left(z, X,Y\right)$, $X=\left(X_1,\dots,X_l\right)$,  such that 
$$
P_j\left(z,\Phi(z), \va_j(z)\right)\equiv 0, \ \forall j=1, \dots ,s,
$$
where $\Phi(z)=\left(\va_1(z),\dots,\va_l(z)\right)$.

Moreover, by  the definition of minimal polynomial 
 $$
 \frac{\de P_j\left(z,X,Y\right)}{\de Y}\left(z,\Phi (z),\va_j(z)\right)\not\equiv 0, \ \forall j=1, \dots ,s.
 $$
  on $D$.
 Thus, by the algebraic version of the existence and uniqueness part of the implicit function theorem, there exist a connected open subset $U\subset D$ with $0\in \ov U$ and Nash algebraic functions 
 $\hat \va_j(z,X)$,  defined in a neighborhood $\hat U$ of $\left\{(z, \Phi(z)) \mid z \in U \right\}\subset \C^n \times \C^{l}$, such that
$$
\va_j(z)=\hat \va_j\left(z,\Phi(z)\right), \ \forall j=1, \dots ,s.
$$
for any $z\in U$.
%
%
Let us denote
$$
\left(\hat f_{1}^{(0)}(z, X),\dots,\hat f_{m_0}^{(0)}(z, X),\dots,\hat f_{1}^{(\ell)}(z, X),\dots,\hat f_{m_\ell}^{(\ell)}(z, X)\right)=(\hat\varphi_1 (z, X), \dots , \hat\varphi_s(z, X)),
$$ 
(notice that, by \eqref{eqfvarf}, $\hat f_{i}^{(k)}(z, \Phi(z))=\hat f_{i}^{(k)}(z)$ for all $k=0, \dots ,\ell$ and $i=1, \dots ,m_k$).
We define
$$
\hat F_k(z,X):=\left(\hat f_{1}^{(k)}(z,X),\dots,\hat f_{m_k}^{(k)}(z,X)\right), k=0, \dots ,\ell.
$$

Consider the function
\begin{equation*}\begin{split} 
 \Psi(z,X,w)&:=
{\tilde\xi _0\left(\hat F_0(z,X), F_0(w) \right)}\\
&-{\mu_1}\log\left(\tilde\xi _1\left(\hat F_1(z,X), F_1(w)\right)\right)-\dots - \mu_\ell\log\left(\tilde\xi _\ell\left(\hat F_\ell(z,X), F_\ell(w)\right)\right),
\end{split}\end{equation*}
where $\tilde\xi _j$ is the real analytic extension of $\xi _j$ in a \ngh\ of $(0,0) \in \C^{m_j} \times \operatorname {Conj} \C^{m_j}$ and $F_k(w)=\left(f_{1}^{(k)}(w),\dots,f_{m_k}^{(k)}(w)\right)$. 
By shrinking $U$ if necessary we can assume $\Psi(z,X,w)$  is defined on  $\hat U \times U$.

We claim that $\Psi(z,X,w)$ vanishes identically on  this set.
Recalling that $\psi _k(z,w)=\tilde \xi _k \left(F_k(z),F_k(w)\right)$ is of diastasis-type, we see that 
$\tilde \xi _k (\hat F_k(z,X),F_k(0))=\psi _k(0)$,
 $k=0,\dots,\ell$. Since $0\in \ov U$, it follows by \eqref{egznew}
that  $\Psi(z,X,0)\equiv 0$.
 Hence, in order to prove the claim, it is enough to show that 
$(\de_w \Psi)(z,X,w)\equiv 0$ for all $w\in U$.  Assume, by contradiction, that there exists $w_0\in U$ such that $(\de_w \Psi)(z,X,w_0)\not\equiv 0$. Since $(\de_w \Psi)(z,X,w_0)$ is Nash algebraic in $(z,X)$ there exists a holomorphic polynomial $P(z,X,t)=A_d(z,X)t^d+\dots+A_0(z,X)$ with $A_0(z,X)\not\equiv 0$ such that  $P(z,X,(\de_w \Psi)(z,X,w_0))=0$. 
Since,  by \eqref{egznew}  and  \eqref{egz} we have
 $\Psi(z,\Phi (z),w)\equiv 0$ we get $(\de_w \Psi)(z,\Phi (z),w)\equiv 0$. Thus  $A_0(z,\Phi (z))\equiv 0$
 which contradicts  the fact that $\va_1(z),\dots,\va_l(z)$ are algebraic independent over $\FR$. Hence $(\de_w \Psi)(z,X,w_0)\equiv 0$ and the claim is proved.

Therefore
\begin{equation*}\begin{split} 
e^{\tilde\xi _0\left(\hat F_0(z,X), F_0(w) \right)} = 
\left(\tilde\xi _1\left(\hat F_1(z,X), F_1(w)\right)\right)^{\mu_1}\cdots \left(\tilde\xi _\ell\left(\hat F_\ell(z,X), F_\ell(w)\right)\right)^{\mu_\ell},
\end{split}\end{equation*}
for every $(z,X,w)\in \hat U \times U$. By fixing  $w\in U$ and  applying \cite[Lemma 2.2]{HYbook} we deduce that 
$\tilde\xi _0\left(\hat F_0(z,X), F_0(w) \right)$
is constant in $(z,X)$.
Thus, by evaluating at $X=\Phi(z)$ one obtains  that 
$\tilde\xi _0\left(\hat F_0(z), F_0(w) \right)$
is constant for fixed $w$, forcing $\psi_0(z)=\tilde\xi _0\left(\hat F_0(z), F_0(z) \right)$  to be constant  for all $z$. The proof of the theorem is complete.
\end{proof}

\section{\K\ potential of homogeneous \K\ metrics on complex bounded domains (by Hideyuki Ishi)}\label{appendix}

Let $\mathcal{D}_{r}$ be the so-called Siegel upper half plane of rank $r$, that is, the set of complex symmetric matrices $z \in \operatorname{Sym}(r, \mathbb{C})$ of size $r \times r$ such that the imaginary part $\Im z$ is positive definite. Note that $\mathcal{D}_{1}=\mathbb{H}$. Let $H_{r}$ be the group of real lower triangular matrices with positive diagonal entries, and define
\begin{equation*}
\mathcal{S}_{r}:=\left\{b(v, T):=\left(\begin{array}{cc}
I_{r} & v \\
& I_{r}
\end{array}\right)\left(\begin{array}{cc}
T & \\
& \mathrm{t}\left(T^{-1}\right)
\end{array}\right) \ |\  v \in \operatorname{Sym}(r, \mathbb{R}), T \in H_{r}\right\}
\end{equation*}
which is a maximal real split solvable Lie subgroup of the real symplectic group $S p(r, \mathbb{R})$. The solvable group $\mathcal{S}_{r}$ acts on $\mathcal{D}_{r}$ simply transitively by
\begin{equation*}
b(v, T) \cdot z:=v+T z^{\mathrm{t}} T ,\quad\left(b(v, T) \in \mathcal{S}_{r}, z \in \mathcal{D}_{r}\right)
\end{equation*}

For a complex symmetric matrix $w \in \operatorname{Sym}(r, \mathbb{C})$ and $k=1, \ldots, r$, we denote by $\Delta_{k}(w)$ the principal $\operatorname{minor} \operatorname{det}\left(\begin{array}{ccc}w_{11} & \ldots & w_{k 1} \\ \vdots & \ddots & \vdots \\ w_{k 1} & \ldots & w_{k k}\end{array}\right)$. For $\underline{s}:=\left(s_{1}, \ldots, s_{r}\right) \in \mathbb{C}^{r}$, define \begin{equation*} \Delta_{\underline{s}}(w):=\Delta_{1}(w)^{s_{1}} \prod_{k=2}^{r}\left(\frac{\Delta_{k}(w)}{\Delta_{k-1}(w)}\right)^{s_{k}}, \end{equation*}
which is called a generalized power function of $w$ \cite[p. 122]{Faraut}. We have the formulas
\begin{enumerate}
\item $\Delta_{\underline{s}}\left(T w^{\mathrm{t}} T\right)=\left(T_{11}^{2 s_{2}} T_{22}^{2 s_{2}} \ldots T_{r r}^{2 s_{r}}\right) \Delta_{\underline{s}}(w) \quad\left(T \in H_{r}, w \in \operatorname{Sym}(r, \mathbb{C})\right)$,
\item $\underline{s}=(\alpha, \ldots, \alpha) \Rightarrow \Delta_{\underline{s}}(w)=(\operatorname{det} w)^{\alpha}$,
\item $w=\operatorname{diag}\left(w_{11}, w_{22}, \ldots, w_{r r}\right) \Rightarrow \Delta_{\underline{s}}(w)=w_{11}^{2 s_{1}} w_{22}^{2 s_{2}} \ldots w_{r r}^{2 s_{r}}$.
\end{enumerate}
For $\underline{\gamma}=\left(\gamma_{1}, \ldots, \gamma_{r}\right) \in \mathbb{R}_{>0}^{r}$, define a positive function $F_{\gamma}: \mathcal{D}_{r} \rightarrow \mathbb{R}_{>0}$ by $F_{\underline{\gamma}}(z):=\Delta_{-\underline{\gamma}}(\Im z)(z \in$ $\left.\mathcal{D}_{r}\right)$. Then it is known that $\log F_{\underline{\gamma}}$ is plurisubharmonic, which means that $\log F_{\underline{\gamma}}$ is a potential function of a Kähler metric $g_{\underline{\gamma}}$ on $\overline{\mathcal{D}}_{r}$. By (1), we have
\begin{equation*}
\log F_{\underline{\gamma}}(b(v, T) z)=\log F_{\underline{\gamma}}(z)+\log \left(T_{11}^{-2 \gamma_{1}} \ldots T_{r r}^{-2 \gamma_{r}}\right) \quad\left(z \in \mathcal{D}_{r}, b(v, T) z\right),
\end{equation*}
so that the metric $g_{\underline{\gamma}}$ is invariant under the action of $\mathcal{S}_{r}$ on $\mathcal{D}_{r}$. We see from Dorfmeister \cite{DORF2} that any $\mathcal{S}_{r}$-invariant Kähler metric on $\mathcal{D}_{r}$ is of the form $g_{\underline{\gamma}}$. Furthermore, any homogeneous Kähler metric $g$ on $\mathcal{D}_{r}$ is equivalent to some $g_{\underline{\gamma}}$, which means that there exists a biholomorphic map $\varphi: \mathcal{D}_{r} \rightarrow \mathcal{D}_{r}$ such that $g=\varphi^{*} g_{\underline{\gamma}}$. It is also known that the biholomorphic map $\varphi$ on the Siegel upper half plane $\mathcal{D}_{r}$ is birational. Note that $F_{\underline{\gamma}}(z)$ is a rational function if and only if all $\gamma_{k}$ are positive integers. Let us introduce a rational function $F_{k}(k=1, \ldots, r)$ defined by
\begin{equation*}
F_{k}(z):= \begin{cases}\Delta_{1}(\Im z) & (k=1) \\ \Delta_{k}(\Im z) / \Delta_{k-1}(\Im z) & (k=2, \ldots, r)\end{cases}
\end{equation*}
Then we have
\begin{equation*}
\log F_{\underline{\gamma}}(z)=\gamma_{1} \log F_{1}(z)+\gamma_{2} \log F_{2}(z)+\cdots+\gamma_{r} \log F_{r}(z) \quad\left(z \in \mathcal{D}_{r}\right) .
\end{equation*}
This argument can be generalized to any homogeneous Siegel domain.

\begin{theor}\label{thmhide}
Let $\Omega$ be a homogeneous Siegel domain and $g$ a homogeneous Kähler metric on $\W$  Then there exist rational functions $F_{1}, \ldots, F_{r}$ and positive numbers $\gamma_{1}, \ldots, \gamma_{r}$ such that $\sum_{k=1}^{r} \gamma_{k} \log F_{k}(z)$ is a potential function of $g$.

In particular, when $g$ is the Bergman metric the coefficients $\gamma_{1}, \ldots, \gamma_{r}$ are positive integers.
\end{theor}

A concrete description of the rational functions $F_{k}(k=1, \ldots, r)$ can be found in Gindikin \cite{GIN}. For the relation between the metric $g$ and the coefficients $\gamma_{1}, \ldots, \gamma_{r}$, see \cite[Proof of Theorem 4]{SIL}.

\section{Proof of Theorem \ref{mainteor}}\label{main}

Given a complex manifold $M$ endowed with a real analytic 
\K \ metric $g$,
Calabi introduced,
in a neighborhood of a point 
$p\in M$,
a very special
\K\ potential
$D^g_0$ for the metric
$g$, which 
he called 
{\em diastasis} (the reader is referred either to \cite{Cal} or to \cite{LoiZedda-book}
for details).

Let $\psi$ be a real analytic \K\ potential for $g$ centred at $p$, 
by duplicating the variables $z$ and $\bar{z}$, the potential $\psi$ can be complex analytically continued to a function $\tilde{\psi}$ defined in a neighborhood $U  \times U$ of $(p, \bar{p}) \in  M \times M$, then the diastasis $D^g_p:U\to\R$ is defined by 
\begin{equation}\label{defdiastasis}
D_p(z)=\tilde\psi(z,\bar z)+ \tilde\psi(p,\bar p) -\tilde\psi(z,\bar p)-\tilde\psi(p,\bar z).
\end{equation}

Among all the potentials the diastasis
is characterized by the fact that 
in every coordinate system 
$\{z_1, \dots , z_n\}$ centered at $p$
\begin{equation}\label{eqcardi}
D^g_p(z, \bar z)=\sum _{|j|, |k|\geq 0}
a_{jk}z^j\bar z^k,
\end{equation}
with 
$a_{j 0}=a_{0 j}=0$
for all multi-indices
$j$.
Clearly the diastasis $D^g_p$ is  a function of diastasis-type as defined in Section \ref{secnash}.

The next general  proposition will be used in the proof of Theorem \ref{mainteor}.
\begin{prop}\label{mainprop}
Let $(M, g)$ be a real analytic  \K\ manifold and let   $\{z_1, \dots , z_n\}$be holomorphic coordinates  centered at a point  $p$, in a neighborhood $U\subset M$ where the diastasis $D^g_p$ is defined. Assume that 
\begin{equation}\label{Dpc=1}
e^{D^g_p}\in \F^{\gamma_1}\cdots \F^{\gamma_{r}},
\end{equation}
for some  $\gamma_1, \dots, \gamma_{r} \in \R$. 
Then any KRS $(g, X)$ on $M$ is trivial, i.e. $g$ is KE.
\end{prop}
\proof
Let us start to write down the KRS equation 
\begin{equation}\label{eqkrsg}
\Ric_{g}=\lambda g+L_{X}g
\end{equation}
in local complex   coordinates $\{z_1, \dots , z_n\}$ for $M$ centred at $p$ in a neighborhood $U$ of $p$ where the diastasis $D^g_p$ is defined.
Since the solitonic vector field $X$ can be assumed to be  the real part of a holomorphic vector field, we can write 
$$
X=\sum_{j=1}^{n}\left(f_{j} \frac{\partial}{\partial z_j}+\bar{f}_{j}\frac{\partial}{\partial\bar z_j}\right)
$$
for some holomorphic functions $f_{j}, j=1, \ldots, n$, on $U$.

Thus, by the definition of Lie derivative, after a 
straightforward computation we can write on $U$
\begin{equation}\label{LXomega}
L_X\omega=\frac{i}{2}\partial\bar\partial f_X.
\end{equation}
where $\omega$ is the \K\ form associated to $g$ and 
\begin{equation}\label{fXlocal}
f_X=\sum_{j=1}^nf_j\frac{\partial D^g_p}{\partial z_j}+\bar f_j\frac{\partial D^g_p}{\partial \bar z_j}.
\end{equation}

By the $\partial\bar\partial$-Lemma one deduces (cfr. \cite{LMpams} for details) that the local expression of \eqref{eqkrsg} is given by 
\begin{equation}\label{detee}
\det\left[\frac{\partial^{2} D^g_{p}}{\partial z_{a} \partial\bar z_{\beta}}\right]=e^{-\frac{\lambda}{2}D^g_p-\frac{f_X}{2}+h+\bar h},
\end{equation}
for a holomorphic function $h$ on $U$. 
Notice that the  function
 $\det\left[\frac{\partial^{2} D^g_{p}}{\partial z_{a} \partial\bar z_{\beta}}\right]$ is finitely generated by rational functions composed with holomorphic or anti-holomorphic functions around $p$.
Moreover it is real valued, since the matrix  $\frac{\partial^{2} D^g_{p}}{\partial z_{a} \bar\partial z_{\beta}}$ is Hermitian. We conclude that $\det\left[\frac{\partial^{2} D^g_{p}}{\partial z_{a} \partial\bar z_{\beta}}\right]\in \mathcal F$.
Furthermore one can check (cf. \cite{LMpams}) that the function $\det\left[\frac{\partial^{2} D^g_{p}}{\partial z_{a} \partial\bar z_{\beta}}\right]$ is of diastasis-type,  thus  $\det\left[\frac{\partial^{2} D^g_{p}}{\partial z_{a} \partial\bar z_{\beta}}\right]\in \F$. 

From \eqref{fXlocal}, one sees 
that  
$$\zeta:=-\frac{f_X}{2}+h+\bar h\in \mathcal F.$$ 
By  \eqref{Dpc=1}, \eqref{detee} and the previous observation one deduces that
\begin{equation}\label{eg2}
e^{\zeta }=
\left[e^{\frac{1}{2}D^g_p}\right]^{{\lambda}}\det\left[\frac{\partial^{2} D^g_{p}}{\partial z_{a} \partial\bar z_{\beta}}\right]\in \F^{\mu_1}\cdots\F^{\mu_r} \F,
\end{equation}
with $\mu_j=\frac{1}{2}{\lambda\gamma_j}$, $j=1,\dots,r$.
On  the one hand  \eqref{eg2} shows that $e^{\zeta }$ and hence $\zeta$  is of diastasis-type  and so, 
$\zeta \in\F$. 
On the other  hand  \eqref{eg2} together  with 
Theorem \ref{lem}  force $\zeta$  to be a  constant and so $f_X$ is the real part of a holomorphic function.
Therefore, by \eqref{eqkrsg} and   \eqref{LXomega}  the metric   $g$ is KE.
\endproof

\begin{proof}[Proof of Theorem \ref{mainteor}]
Let  $\varphi:\left(M, g\right)\rightarrow \left(\W, g_{\W}\right)$ be a holomorphic isometry. Let $\{z_1, \dots , z_n\}$ be a coordinate system for $M$ centered at $p$, 
by the hereditary property of Calabi's diastasis function (see \cite{Cal} for details), we have that
\begin{equation}\label{eqdipr}
D^{g}_p(z)=D^{g_\W}_{\varphi (p)}(\vf(z)).
\end{equation}

Consider the realization of $\W$ as a homogeneous Siegel domain as in the previous section. 
From
the definition of Calabi's diastasis function \eqref{defdiastasis} and 
 Theorem \ref{thmhide}  we deduce  that the 
the diastasis function for the homogeneous  metric $g_\Omega$  at given point $v\in\Omega$ is 
given by:
\begin{equation}\label{diastomega}
D^{g_\Omega}_v(u)=\sum_{k=1}^r\gamma_k
\log\left(\frac{F_{k}(u,\ov u)\hs F_{k}(v,\ov v)}{F_{k}(u,\ov v)\hs F_{k}(v,\ov u)}\right),
\end{equation}
where $\gamma_k$ are positive real numbers and  $F_k(u, \ov v )$ are rational holomorphic  functions
in  $u$ and $\ov v$, $k=1, \dots ,r$.
Therefore
\begin{equation}\label{eqeelg}
D^{g_\W}_{\varphi (p)}(\varphi (z))=\sum_{k=1}^{r} \gamma_{k} \log \left(\frac{F_{k}(\varphi  (z),\ov {\varphi  (z)})\hs F_{k}(\varphi  (0),\ov {\varphi  (0)})}{F_{k}(\varphi  (z),\ov {\varphi  (0)})\hs F_{k}({\varphi  (0)},\ov {\varphi  (z)})}\right)
\end{equation}
From \eqref{eqdipr} and \eqref{eqeelg}, we get
$$ 
e^{D^{g}_p(z)}=\prod_{k=1}^{r}\left(\frac{F_{k}(\varphi (z),\ov {\varphi (z)})\hs F_{k}({\varphi (0)},\ov {\varphi (0))}}{F_{k}({\varphi (z)},\ov {\varphi (0)})\hs F_{k}({\varphi (0)},\ov {\varphi (z)})}\right)^{\gamma_k}.
$$
and hence
$$
e^{D^{g}_p(z)}\in  \mathcal F^{\gamma_1}\cdots \mathcal F^{\gamma_{r}}.
$$ 
On the other hand, from \eqref{eqcardi}, we can see that each term in \eqref{eqeelg} 
is of diastasis-type. That  is
$$
e^{D^{g}_p(z)} \in \F^{\gamma_1}\cdots \F^{\gamma_{r}},
$$
namely \eqref{Dpc=1}.
Thus, by Proposition \ref{mainprop} we deduce that $g$ is KE,
completing the proof of \emph{(i)}  of Theorem \ref{mainteor}.
 
\noindent\emph{Proof of (ii).} 
Let $(\W,g_\W)$ be a homogeneous bounded domain and let $U \subset \C$ be a \ngh\ of the origin and let $g_0$ the indefinite flat metric on $\C^n$ whose \K\ potential is given by $|z|_l^2:=|z_1|^2+\dots+|z_l|^2-|z_{l+1}|^2-\dots-|z_n|^2$, $l\in\left\{0,1,\dots,n\right\}$. Assume by contradiction that there exists a holomorphic immersion $\varphi:U \to  \W$ and $\eta:U\to \C^n$, $\eta (0)=0$, such that 
$$
\varphi^* g_\W  = \eta^* g_0. 
$$
By \eqref{eqdipr}  and \eqref{eqeelg} this is equivalent to the following equation
$$
\sum_{k=1}^{r} \gamma_{k} \log \left(\frac{F_{k}(\varphi (z),\ov {\varphi (z)})\hs F_{k}(\varphi (0),\ov{\varphi (0)})}{F_{k}(\varphi (z),\ov{\varphi (0)})\hs F_{k}(\varphi (0),\ov{\varphi (z)})}\right) = |\eta (z)|_l^2.
$$
So that
$$
e^{ |\eta(z)|_l^2}=\prod_{k=1}^{r} \left(\frac{F_{k}(\varphi (z),\ov{\varphi (z)}\hs F_{k}(\varphi (0),\ov{\varphi (0))}}{F_{k}(\varphi (z),\ov{\varphi (0)})\hs F_{k}(\varphi (0),\ov{\varphi (z)})}\right)^{\gamma_{k}} .
$$
By applying Proposition \ref{mainprop} and Theorem \ref{lem} we see that $\eta$ is forced to by constant, which contradicts the fact that  $\eta$ is an immersion (unless $M$ is zero-dimensional). 
\end{proof}

\end{document}